\title{Higher-Order Autocorrelations on Finite Abelian Groups}
\author{Aaron Agulnick, Toby Busick-Warner}
\date{April 14, 2026}
\DocumentMetadata{testphase=new-or-1} % supress links when they cross page boundaries
\documentclass[10pt]{article} 

\usepackage{amsmath}                                  % math
\usepackage{amssymb}                                  % symbols
\usepackage{amsfonts}                                 % fonts
\usepackage{amsthm}                                   % theorem environments
\usepackage{mathtools}                                % nice looking sum limits
\usepackage[backend=biber,style=alphabetic]{biblatex} % bibliography
\usepackage{hyperref}                                 % links
\usepackage{microtype}                                % microtypography
\usepackage{thmtools}
\usepackage{thm-restate}
\usepackage{cleveref}

\theoremstyle{plain}
\newtheorem{theorem}{Theorem}[section]

\newtheorem{lemma}[theorem]{Lemma}
\newtheorem{proposition}[theorem]{Proposition}
\theoremstyle{definition}
\newtheorem{definition}[theorem]{Definition}
\newtheorem{example}[theorem]{Example}

\newcommand{\CC}{\mathbb{C}}
\newcommand{\QQ}{\mathbb{Q}}
\newcommand{\RR}{\mathbb{R}}
\newcommand{\ZZ}{\mathbb{Z}}
\newcommand{\NN}{\mathbb{N}}
\newcommand{\bangle}[1]{\langle {#1} \rangle}
\newcommand{\conj}[1]{\overline{#1}}
\newcommand{\dual}[1]{\widehat{#1}}

\DeclareMathOperator{\Gal}{Gal}

\DeclareMathOperator{\tr}{tr}

\begin{document}
\maketitle

\begin{abstract}
    The question of determining a signal from its higher-order autocorrelation data is of practical interest in fields as varied as X-ray crystallography, image processing, and satellite communications. At the heart of the issue is how much of this autocorrelation data one truly needs. We prove two new upper bounds on the order of data needed to determine a signal on a general (i.e. not necessarily cyclic) finite abelian group depending on some knowledge of the vanishing of the signal's Fourier transform. In investigating lower bounds on the required data, we classify signals on \(\ZZ_6\) \textbf{not} determined by their fifth-order data and provide analogous examples on \(\ZZ_{30}\).
\end{abstract}

\section{Introduction}
For a locally compact abelian group \(G\) and a function \(f \in L^1(G,\CC)\), the \(n\)th-order autocorrelation of \(f\) is a function 
\begin{equation} \label{def: autocorrelation}
    \rho_n(f)(t_1, \dotsc, t_{n-1}) = \int_G f(x)f(x+t_1)\dotsm f(x+t_{n-1})dx.
\end{equation}

It is often more natural to consider on the Fourier side. An elementary computation with Fubini's theorem shows 
\begin{equation} \label{autocorrelation fourier transform}
    \dual{\rho}_n(f)(\chi_1, \dotsc, \chi_{n-1}) = \hat{f}(\chi_1)\hat{f}(\chi_2)\dotsm\hat{f}(\chi_{n-1})\hat{f}(\chi_1^{-1}\chi_2^{-1} \dotsm \chi_{n-1}^{-1}).
\end{equation}
We may instead interpret \(\dual{\rho}_n(f)\) as a function on \(n\) characters whose product is the identity.

Low-order autocorrelation data is already a frequently interesting statistic: 
the values of the second-order autocorrelation \(\rho_2(1_E)\) of the characteristic function of a set \(E\) contains exactly the same information as the difference multiset \(E-E\), 
and on the Fourier side, \(\dual{\rho}_2(f)(e_k) = \hat{f}(e_k)\hat{f}(e_{-k}) = |\hat{f}(e_k)|^2\) when \(f\) is real-valued. 
In particular, this data is familiar to X-ray crystallographers as precisely the statistic one measures as the diffraction pattern of an X-ray shone through a crystal. 

The central question in both of these applications is the same: 
how can \(f\) be reconstructed from its autocorrelation data?
It is clear from translation-invariance of the integral in (\ref{def: autocorrelation}) that  \(\rho_n(f)\) is a \emph{translation-invariant}: 
if there exists \(a\in G\) such that \(f(x) = g(x+a)\) for all \(x\), then \(\rho_n(f) = \rho_n(g)\) for all \(n\). 
But \(\dual{\rho}_2(f) = |\hat{f}|^2\) reveals a more serious issue, the so-called \emph{phase problem} of X-ray crystallography: 
second-order data offers only the magnitude of the complex-valued \(\hat{f}\), not its phase, and hence is insufficient to determine \(f\) in any reasonable case. 
Indeed, Pauling and Shappell \cite{pauling_8_1930} discovered two forms of the mineral bixbyite with different crystalline structures which are not distinguished by second-order autocorrelation data. 
This provided a counterexample to a claim of Patterson \cite{patterson_direct_nodate}, sparking a long-standing interest in the classification of such counterexamples, called \emph{(strictly) homometric sets}. 
A full account of this story can be found in \cite{moore_patterson_1989}, and an algebraic approach to this classification appears in \cite{rosenblatt_phase_1984}. 

Despite the issue of homometric structures, useful data can still be extracted from the autocorrelations of \(f\). 
In satellite communication, third-order autocorrelations are experimentally extracted from an ultrashort laser pulse \cite{paulter_new_1991} and used to infer information about such signals even with equipment which is necessarily imprecise so as to withstand such high-intensity electromagnetic fields \cite{trebino_highly_2020}. 
In X-ray crystallography, statistical inferences about third-order autocorrelation data were used by Hauptman \cite{hauptman_phase_1991} to solve (certain instances of) the phase problem, a much-celebrated result. 

The success of third-order autocorrelations in applications is owed to the following fact: 
\begin{proposition}
    If \(f\) and \(g\) are \(\RR\)-valued functions on a locally compact abelian group \(G\) which have almost-nowhere-vanishing Fourier transforms and share autocorrelation data through order three, i.e. \(\rho_n(f) = \rho_n(g)\) for \(n=1,2,3\), then \(f(x) = g(x+a)\) for some \(a \in G\) and all \(x \in G\).
\end{proposition}

In light of results like the above, whenever \(\rho_i(f) = \rho_i(g)\) for all \(i \in I \subseteq \NN\) implies that \(f\) and \(g\) differ by a translation, we say that \(\{\rho_i: i \in I\}\) is a \emph{complete set} of translation invariants for \(f\). The core of the issue is the vanishing of the Fourier transform, which is already clearly a source of data loss in the expression (\ref{autocorrelation fourier transform}), so we make the following central definition: \begin{definition}
    The support of a function $f: G \to \CC$ is \begin{equation*}
        S(f) := \{x \in G: \hat{f}(x) \neq 0\}.
    \end{equation*}
\end{definition}
The above proposition is a shadow of the following theorem, which is more complicated in its statement but vastly more powerful: \begin{theorem} \label{adlerkonheim}
    \textnormal{\cite[p.\ 428]{adler}}
    For \(G\) a locally compact abelian group and \(f,g\in L^1(G,\RR)\), if every element of the generated subgroup \(\bangle{S(f)} \leq \dual{G}\) is a product of at most \(v\) elements of \(S(f)\), 
    then \(\{\rho_n : n = 1, \dotsc, 3v\}\) is a complete set of translation invariants.
\end{theorem}

For applications, a signal on \(\RR^3\) can be discretized into a signal on \(\ZZ_{N_1} \times \ZZ_{N_2} \times \ZZ_{N_3}\), and the granularity of this discretization can be chosen so that \(N_1, N_2, N_3\) are pairwise relatively prime and hence \(\ZZ_{N_1} \times \ZZ_{N_2} \times \ZZ_{N_3} \cong \ZZ_{N_1 N_2 N_3}\). 
However, despite the successes of third-order autocorrelation data in this setting for numerous applications, it is shown in \cite{grunbaum} that even in the case of a finite cyclic domain, data through order three is, in general, insufficient. 
They prove the following theorem: 
\begin{theorem} \label{grunbaum-moore}
    \textnormal{\cite[p.\ 318]{grunbaum}}
    For \(f,g: \ZZ_N \to \QQ\), \(\{\rho_n : n = 1,2,3,4,5,6\}\) is a complete set of invariants.
    For \(N\) odd, \(\{\rho_n : n = 1,2,3,4\}\) is a complete set.
\end{theorem}

Furthermore, \cite[pp.\ 317-318]{grunbaum} contains a number of examples showing sharpness of this bound. 
Our sections \ref{Z6} and \ref{Z30} offer more such examples, and our Theorem \ref{mainthm:A} classifies all such examples on \(\ZZ_6\). 

As is evident from Theorem \ref{adlerkonheim}, the core of the issue is the vanishing of \(\hat{f}\). 
In searching to understand the vanishing of \(\hat{f}\), one might view \(\hat{f}(e_k)\) as a polynomial in a \(k\)th root of unity \(\omega_k = e_k(1) = e^{2\pi i k/N}\): 
\[\hat{f}(e_k) = P(\omega_k) = \sum_{n=0}^{N-1}f(g)\omega_k^n.\]

Consequently, \(\hat{f}\) vanishes on all or none of the elements of the same (additive) order in \(\ZZ_N\) at once.
In fact, we have the following vast generalization of this observation, arising from (and explaining) the (well-known) fact that \(f\) is real-valued if and only if \(\hat{f}(e_{-k}) = \overline{\hat{f}(e_k)}\). 

\begin{theorem}[Galois action] \label{result}
    Let \(G\) be a finite abelian group, \(f: G \to \CC\), and \(K\) be any subfield of \(\CC\).
    Then \(f(x) \in K\) for all \(x \in G\) if and only if, for all \(e \in \dual{G}\) with \(M = |e|\), we have \(\hat{f}(e) \in L \coloneqq K(\zeta_M)\), 
    and \(\sigma_u(\hat{f}(e)) = \hat{f}(e^u)\) for each \(\sigma_u \in \Gal(L/K) \hookrightarrow \ZZ_M^\times\).
\end{theorem}
\begin{proof} 
    We use the embedding \(\sigma_u \mapsto u \in \ZZ_M^\times\) such that \(\sigma_u(\zeta) = \zeta^u\) for all \(M\)th roots of unity \(\zeta\).

    First, assume \(f(x) \in K\).
    For \(a \in \ZZ\), \((e^a)(x) = (e(x))^a\) by the operation of the dual group.
    We let \(M=|e|\), so \((e(x))^M = (e^M)(x) = \mathbf{1}(x) = 1\) and \(e(x)\) is an \(M\)th root of unity.
    Since \(f(x)\in K\) and \(\conj{e(x)}\) is a power of \(\zeta_M\), we have 
    \[\hat{f}(e) = \sum_{x \in G} f(x) \conj{e(x)} \in L.\]
    Then, 
    \begin{align*}
        \hat{f}(e^u) &= \sum_{x \in G} f(x) \conj{(e^u)(x)}\\
        &= \sum_{x \in G} f(x) \conj{e(x)}^u\\
        &= \sum_{x \in G} \sigma_u(f(x)) \sigma_u(\conj{e(x)})\\
        &= \sigma_u(\hat{f}(e))
    \end{align*}
    because \(\sigma_u\) acts as the identity on \(K\) and raises \(M\)th roots of unity to the \(u\)th power.

    Now, assume that \(\hat{f}(e^u) = \sigma_u(\hat{f}(e))\) for each \(\sigma_u \in \Gal(L/K)\).
    Recall Fourier inversion:
    \[
        f(x) = \frac{1}{|G|}\sum_{e\in \dual{G}} \hat{f}(e)e(x).
    \]
    For any \(e \in \dual{G}\), consider the group action \(\sigma_u \cdot e \coloneqq e^u\) on the set of generators of \(\bangle{e}\)
    (this is the action of the Galois group on the roots of unity pulled through the isomorphism \(\mu_n \cong \dual{G}\)).
    Then consider the sum over the orbit \(\{e^u: \sigma_u \in \Gal(L/K)\}\):
    \begin{align*}
           \sum_{\mathclap{\sigma_u \in \Gal(L/K)}} \hat{f}(e^u)[(e^u)(x)]
        &= \sum_{\mathclap{\sigma_u \in \Gal(L/K)}} \sigma_u(\hat{f}(e))e(x)^u\\
        &= \sum_{\mathclap{\sigma_u \in \Gal(L/K)}} \sigma_u(\hat{f}(e)e(x))\\
        &= \tr_{L/K} (\hat{f}(e)e(x)) \in K
    \end{align*}
    This sum is in \(K\) because it is invariant under the action of \(\Gal(L/K)\).
    Then, we partition \(\dual{G}\) first by the order of each element, then by which orbit each element falls into.
    Since the sum over each orbit is in \(K\), \(f(x) \in K\) for all \(x \in G\).
\end{proof}

During the preparation of this paper, the above theorem appeared independently in \cite[Theorem II.7]{casper} in the special case \(K = \QQ\), where it is used to study higher-order autocorrelations for functions with noncyclic domain.

We will use Theorem \ref{result} in two ways: 
(1) to investigate and ultimately classify examples of functions on cyclic groups which are \emph{not} determined by their low-order autocorrelation data, using techniques of algebraic number theory, and 
(2) to generalize theorems \ref{adlerkonheim} and \ref{grunbaum-moore} to the case where the domain \(G\) is not necessarily cyclic. 
We will prove the following three results: 
\begin{restatable}{mainthm}{thmA}
\label{mainthm:A}
    Let \(f,g : \ZZ_6 \to \QQ\).
    \(\rho_n(f) = \rho_n(g)\) for \(n = 1, 2, 3, 4, 5\) but not for \(n=6\) if and only if 
    \begin{itemize}
        \item \(\hat{f}(e_0) = \hat{g}(e_0)\),
        \item \(\hat{f}(e_k) = \hat{g}(e_k) = 0\) for \(k = 2, 3, 4\),
        \item \(|\hat{f}(e_1)| = |\hat{g}(e_1)|\), and
        \item \(\hat{f}(e_1)^6 \ne \hat{g}(e_1)^6\).
    \end{itemize}
    Up to scaling, such pairs of functions correspond to natural numbers \(L = |\hat{f}(e_1)|^2\) with more than six distinct factorizations in \(\ZZ[\omega_3]\), or equivalently to \(L \in \NN\) with at least one prime factor that is 1 modulo 3.
\end{restatable}

\begin{example}
The functions on \(\ZZ_{30}\) given by 
\begin{align*}
    f = [&56, -7, 7, 14, 7, 28, -14, -7, 7, 14, -28, -7, -14, -7, 7, \\
        -&56, 7, -7, -14, -7, -28, 14, 7, -7, -14, 28, 7, 14, 7, -7];\\
    g = [&52, -2, 11, 13, 2, 44, -13, -2, 11, 13, -8, -11, -13, -2, 11, \\
        -&52, 2, -11, -13, -2, -44, 13, 2, -11, -13, 8, 11, 13, 2, -11].
\end{align*}
share autocorrelation data through order 5, but are not translates of each other.
\end{example}

\begin{restatable}{mainthm}{thmB}
\label{mainthm:B}
    For \(G\) a locally compact abelian group and \(f\in L^1(G,\RR)\), suppose \(S(f)^v\) is an index 2 subgroup of \(\bangle{S(f)}\).
    Then \(\{\rho_n : n = 1,\dotsc, 3v\}\) is a complete set of translation invariants.
\end{restatable}

We will extract from Theorem \ref{mainthm:B} an immediate proof of Theorem \ref{grunbaum-moore}.

\section{Examples on \texorpdfstring{\(\ZZ_6\)}{ℤ₆}} \label{Z6}
An immediate application is the construction of examples of pairs of functions \(\ZZ_6 \to \QQ\) not distinguished by their low-order autocorrelation data. 
Gr\"unbaum and Moore construct a handful of examples in \cite[Section 6]{grunbaum}; 
Casper and Orozco give an infinite family generalizing those of Gr\"unbaum and Moore in \cite[Example IV.1]{casper}.
We present an slightly larger family, give an number-theoretic method of constructing concrete examples, and indeed prove that we have a complete classification of such pairs of functions. We write $e_k$ for the character of $\ZZ_N$ corresponding to $k \in \ZZ_N$, namely $x \mapsto e^{2\pi i k x/N}$.

\thmA*

\begin{proof}
	\(\rho_n(f) = \rho_n(g) \iff \dual{\rho}_n(f) = \dual{\rho}_n(g)\), so it is natural to work with the values of \(\hat{f}\) and \(\hat{g}\).
	We again will treat \(\dual{\rho}_n(f)\) as a function on \(n\) characters whose product is the identity.

    Suppose \(\rho_n(f) = \rho_n(g)\) for \(n = 1, 2, 3, 4, 5\) but not for \(n=6\).
	First, 
    \[\hat{f}(e_0) = \rho_1(f)(e_0) = \rho_1(g)(e_0) = \hat{g}(e_0).\]

    Then compute all values of \(\dual{\rho}_6(f)\) and \(\dual{\rho}_6(g)\).
    One may check that among all lists of six characters whose product is the identity, all lists can be broken up into two lists, each of whose product is the identity, except for \((e_1,e_1,e_1,e_1,e_1,e_1)\) and \((e_5,e_5,e_5,e_5,e_5,e_5)\).
    This means that for any other such list \(A\), separated this way into lists \(B\) and \(C\) of lengths \(b\) and \(c\), we have
    \[\dual{\rho}_6(f)(A) = \dual{\rho}_b(f)(B)\dual{\rho}_c(f)(C) = \dual{\rho}_b(g)(B)\dual{\rho}_c(g)(C) = \dual{\rho}_6(g)(A).\]
    Thus, \(\dual{\rho}_6(f)\) and \(\dual{\rho}_6(g)\) must differ at \((e_1,e_1,e_1,e_1,e_1,e_1)\) (and equivalently at \((e_5,e_5,e_5,e_5,e_5,e_5)\)).
    Therefore,
    \begin{equation} \label{6th differ}
        \hat{f}(e_1)^6 = \dual{\rho}_6(f)(e_1,e_1,e_1,e_1,e_1,e_1) \ne \dual{\rho}_6(g)(e_1,e_1,e_1,e_1,e_1,e_1) = \hat{g}(e_1)^6.
    \end{equation}
    We also have 
    \[|\hat{f}(e_1)|^2 = \hat{f}(e_1)\hat{f}(e_5) = \rho_2(f)(e_1,e_5) = \rho_2(g)(e_1,e_5) = \hat{g}(e_1)\hat{g}(e_5) = |\hat{g}(e_1)|^2,\]
    so \(|\hat{f}(e_1)| = |\hat{g}(e_1)|\).

    If \(e_3 \in S(f)\), then using \(e_3^2 = e_1^3 e_4 = 1\),
    \begin{align*}
        \hat{f}(e_3)^2 = \dual{\rho}_2(f)(e_3, e_3) &= \dual{\rho}_2(f)(e_3, e_3) = \hat{g}(e_3)^2 \ne 0,\\
        \text{and }\hat{f}(e_1)^3\hat{f}(e_3) = \dual{\rho}_4(f)(e_1,e_1,e_1,e_3) &= \dual{\rho}_4(g)(e_1,e_1,e_1,e_3) = \hat{g}(e_1)^3\hat{g}(e_3),\\
        \implies \hat{f}(e_1)^6\hat{f}(e_3)^2 &= \hat{g}(e_1)^6\hat{g}(e_3)^2,\\
        \implies \hat{f}(e_1)^6 &= \hat{g}(e_1)^6,
    \end{align*}
    contradicting (\ref{6th differ}).
	If \(e_4 \in S(f)\) (and equivalently \(e_2 \in S(f)\)), then using \(e_4^3 = e_1^2 e_4 = 1\),
    \begin{align*}
        \hat{f}(e_4)^3 = \dual{\rho}_3(f)(e_4,e_4,e_4) &= \dual{\rho}_3(g)(e_4,e_4,e_4) = \hat{g}(e_4)^3 \ne 0,\\
        \text{and }\hat{f}(e_1)^2\hat{f}(e_4) = \dual{\rho}_3(f)(e_1,e_1,e_4) &= \dual{\rho}_3(g)(e_1,e_1,e_4) = \hat{g}(e_1)^2\hat{g}(e_4),\\
        \implies \hat{f}(e_1)^6\hat{f}(e_4)^3 &= \hat{g}(e_1)^6\hat{g}(e_4)^3,\\
        \implies \hat{f}(e_1)^6 &= \hat{g}(e_1)^6,
    \end{align*}
    again contradicting (\ref{6th differ}).
    Therefore \(e_2, e_3, e_4 \notin S(f)\), so \(\hat{f}(e_k) = \hat{g}(e_k) = 0\) for \(k = 2, 3, 4\).

    Now proving the reverse direction, the only potentially nonvanishing lists of inputs \(A\) for \(\dual{\rho}_n(f)\) and \(\dual{\rho}_n(g)\) for \(n \leq 5\) have an equal number of \(e_1\) and \(e_5\).
    So, 
    \begin{gather*}
        \dual{\rho}_n(f)(A) = \hat{f}(e_0)^r (\hat{f}(e_1) \hat{f}(e_5))^s = \hat{f}(e_0)^r |\hat{f}(e_1)|^{2s} =\\
        \hat{g}(e_0)^r |\hat{g}(e_1)|^{2s} = \hat{g}(e_0)^r (\hat{g}(e_1) \hat{g}(e_5))^s = \dual{\rho}_n(g)(A).
    \end{gather*}

    Finally, we have already seen that \(\hat{f}(e_1)^6 \ne \hat{g}(e_1)^6\) guarantees that \(\dual{\rho}_6(f) \ne \dual{\rho}_6(g)\).

    We continue to the second part of the theorem: computing examples of such pairs of functions.
    By Theorem \ref{result}, we have \(\hat{f}(e_0) = \hat{g}(e_0) \in \QQ\) and \(\hat{f}(e_1), \hat{g}(e_1)\in\QQ(\zeta_6)\).
    When finding values for \(\hat{f}(e_1)\), we may instead ``clear denominators'' and work over the ring of integers \(\mathcal{O}_{\QQ(\zeta_6)} = \ZZ[\zeta_6]\) by multiplying \(\hat{f}(e_1)\) and \(\hat{g}(e_1)\) by an appropriate integer.
    We let \(\hat{f}(e_1) = x+y\zeta_6\) for integers \(x,y\), and compute 
    \[|\hat{f}(e_1)|^2 = x^2 + xy + y^2.\]
    If we set \(|\hat{f}(e_1)|^2 = L \in \NN\), each solution for \(x+y\zeta_6\) lies in a set of six, differing by sixth roots of unity.
    Since we require \(\hat{f}(e_1)^6 \ne \hat{g}(e_1)^6\), we want to find integers \(L\) with more than six unique integer solutions to \(L = x^2 + xy + y^2\). 
    The solution to this problem lies in the domain of algebraic number theory:
    \begin{proposition}
        \textnormal{\cite[p.\ 33]{dummitNT2}}
        If \(L = 3^k r s\), where \(r\) is a product of primes \(1\) modulo \(3\) and \(s\) is a product of primes \(2\) modulo \(3\), then 
        \(L = x^2 + xy + y^2\) has a solution in integers \(x,y\) if and only if \(s\) is a square.
        In this case, the number of solutions \((x,y)\) is 6 times the number of positive divisors of \(r\).
    \end{proposition}
    Thus, we can simply choose an integer \(r>1\) that is a product of one or more primes \(1\) modulo \(3\) and find its factorization over \(\mathcal{O}_{\QQ(\zeta_6)}\) to get candidates for \(\hat{f}(e_1)\).
\end{proof}

\begin{example}
    The smallest ``interesting'' integer-valued example (where \(f\) is not just a reflection of \(g\)), is the pair
    \begin{align*}
        f&=[13, 11, -2, -13, -11,  2];\\
        g&=[14,  7,  -7, -14, -7,  7].
    \end{align*}
    with \(\hat{f}(e_1) = 39-9\sqrt{-3}, \hat{g}(e_1) = 42\), found by choosing \(r=7\) and clearing denominators.
\end{example}

\section{Examples on \texorpdfstring{\(\ZZ_{30}\)}{ℤ₃₀}}\label{Z30}

We now present a set of examples on \(\ZZ_N\) with \(N \ne 6\). 
Such examples are hard to come by, as \(\varphi(6) = 2\) means it is especially easy to construct functions on \(\ZZ_6\) with small support. Here we take \(N = 30\). 

Similarly to the above functions with \(S(f) = \{e_1, e_5\} \subseteq \dual{\ZZ}_6\), we let the support contain \(e_1\) and the characters implied by it under the Galois action, i.e. the units modulo 30: 
\[S(f) = \{e_1,e_7,e_{11},e_{13},e_{17},e_{19},e_{23},e_{29}\}.\]
If we reuse our values of \(\hat{f}(e_1), \hat{g}(e_1) \in \QQ(\zeta_6)\), we can ask what Theorem \ref{result} requires for \(f,g\) to be rational-valued.
Since \(\zeta_6\) is a sixth root of unity, \(\sigma_1, \sigma_7, \sigma_{13}, \sigma_{19}\) all fix \(\QQ(\zeta_6)\) 
(we actually could have seen this from our earlier discussion computing the Galois groups of intermediate cyclotomic fields).
The other four automorphisms each send \(\zeta_6\) to \(\zeta_6^5 = \conj{\zeta_6}\).
Therefore, we have 
\begin{align*}
    \hat{f}(e_1) = \hat{f}(e_7) &=\hat{f}(e_{13})= \hat{f}(e_{19})\\
    \hat{f}(e_{11}) = \hat{f}(e_{17}) &=\hat{f}(e_{23})= \hat{f}(e_{29}) = \conj{\hat{f}(e_1)}
\end{align*}
Notice that each of the \(e_u\) with \(u \equiv 5\) modulo 6 is the conjugate of the \(e_u\) for \(u\equiv\) 1 modulo 6.
Therefore, just as in the case of \(\ZZ_6\), any list of at most 5 elements from \(S(f)\) whose product is \(1\) 
  must have an equal number of characters whose Fourier coefficient is \(\hat{f}(e_1)\) as characters with coefficient \(\conj{\hat{f}(e_1)}\).
To guarantee that \(\rho_n(f) = \rho_n(g)\) up to order 5, we again only need \(|\hat{f}(e_1)|^2=|\hat{g}(e_1)|^2\).
And once again, when \(n=6\), we can compute \[\rho_6(f)(e_1, e_1, e_1, e_{19}, e_{19}, e_{19}) = \hat{f}(e_1)^6.\]
Therefore, we have exactly the same conditions on \(\hat{f}(e_1)\) and \(\hat{g}(e_1)\) and may reuse the previous solution.
Using the same choices as the explicit example above and clearing denominators, we get the pair of functions
\begin{align*}
    f = [&56, -7, 7, 14, 7, 28, -14, -7, 7, 14, -28, -7, -14, -7, 7, \\
        -&56, 7, -7, -14, -7, -28, 14, 7, -7, -14, 28, 7, 14, 7, -7];\\
    g = [&52, -2, 11, 13, 2, 44, -13, -2, 11, 13, -8, -11, -13, -2, 11, \\
        -&52, 2, -11, -13, -2, -44, 13, 2, -11, -13, 8, 11, 13, 2, -11].
\end{align*}
Unlike in the case of \(\ZZ_6\), this construction does not account for all such pairs of functions, and we conjecture that it is not even a complete classification on this particular support; this is because we are not using the largest field we have access to on this support, namely \(\QQ(\zeta_{30})\).
Other examples of such functions can instead be found using the support generated via the Galois action by \(e_3\) and \(e_5\):
\[S(f) = \{e_3, e_5, e_9, e_{21}, e_{25}, e_{27}\}.\]
Gr\"unbaum and Moore \cite[p.\ 318]{grunbaum} give a single example on this support, while Casper and Orozco \cite[Example IV.4]{casper} give a (much) more general family on \(G = \ZZ_{2pq}^r\) for odd primes \(p,q\), which for \(p=3, q=5, r=1\) specifies to this same support.

\section{New Bounds}
Here we take up the mantle of proving bounds on the orders of autocorrelation data needed to determine certain families of functions. 
The motivating observation is one made by Gr\"unbaum and Moore: 

\begin{proposition}
    \label{sumofunits}
    For \(N>1\) odd or a power of 2, every element of \(\ZZ_N\) is the sum of at most two units.
    Otherwise, each is the sum of at most three units.
\end{proposition} 
\begin{proof}
    We begin by proving that every even residue is the sum of two units: \(\ZZ_N^\times + \ZZ_N^\times = 2\ZZ_N\).
    First, consider \(N=2^k\).
    Every odd residue is relatively prime to the modulus, hence is a unit.
    Thus, for any \(2c \in 2\ZZ_N\) we have \(2c = 1 + (2c-1)\) as a sum of two units.
    Next, let \(N=p^k\) for \(p\) an odd prime.
    Since \(N\) is odd, every residue has an even representative: \(2\ZZ_N = \ZZ_N\).
    Thus, we can take any \(c \in \ZZ_N = 2\ZZ_N\).
    At least one of \(c-1\) or \(c-2\) is a unit since if neither were, then \(p\) would divide their difference, which is \(1\).
    Thus either \(c = 1+(c-1)\) or \(c = 2+(c-2)\) is a sum of two units.

    For a more general \(N = p_1^{k_1} \dotsm p_n^{k_n}\), we let \(2c \in 2\ZZ_N\).
    Then the above discussion proves that for each \(i\) there exist \(a_i, b_i \in \ZZ^\times_{p_i^{k_i}}\) such that \(a_i + b_i \equiv 2c\).
    By the Chinese Remainder Theorem, there exists \(a,b \in \ZZ_N\) such that
    \begin{align*}
        a &\equiv a_i \mod{p_i^{k_i}}\\
        b &\equiv b_i \mod{p_i^{k_i}}
    \end{align*} 
    for all \(i\).
    Then, since \(a_i + b_i \equiv 2c \mod{p_i^{k_i}}\) for each \(i\), we also have \(a + b \equiv 2c \mod{p_i^{k_i}}\), so by the CRT again, \(a+b \equiv 2c \mod{N}\).
    If \(a\) were not a unit modulo \(N\), then some prime \(p_i\) would divide \(a\). 
    But then \(a_i\) would not be a unit modulo \(p_i^{k_i}\), contradicting our assumption.
    Therefore, \(a\) is a unit modulo \(N\), and for the same reason, so is \(b\).

    We have already seen that for \(N=2^k\), every element is either a unit or a sum of two units.
    Again, when \(N\) is odd, every residue has an even representative, and is thus a sum of two units.
    Otherwise, every odd residue is \(1\) plus an even residue, and is thus the sum of at most three units.
\end{proof}

The utility is in tandem with Theorem \ref{adlerkonheim} of Adler and Konheim, and Theorem \ref{result} on the Galois action. 
We obtain immediately the following rather concrete result concerning a finite number of autocorrelations and noncyclic groups. 
It improves the upper bound appearing in the main result of \cite{casper} when some knowledge of \(S(f)\) is assumed.

\begin{theorem}
    \label{9r}
    Let \(G\) be a finite abelian group with invariant factors \(N_1 | \dotsb | N_r\), and let \(k\) be the number of \(N_i\) which are odd or a power of 2.
    Then let \(f : G \to \QQ\), and suppose \(S(f)\) contains the \(r\) ``basis elements'' \(x_i = (1, \dotsc, 1, e_1, 1, \dotsc, 1)\), each from the \(\dual{\ZZ}_{N_i}\) component of \(\dual{G}\).
    The set \(\{\rho_n: n = 1, \dotsc, 3(3r-k)\}\) is a complete set of invariants, and so determines \(f\) up to a shift.
\end{theorem}
\begin{proof}
    By Theorem \ref{result}, every tuple \((k_1, \dotsc, k_r) \in G \cong \bigoplus_{i=1}^r \ZZ_{N_i}\) where each \(k_i\) is a unit in \(\ZZ_{N_i}\) is contained in \(S(f)\) whenever the standard basis elements \((1, 0, \dotsc, 0), \dotsc, (0, \dotsc, 0, 1)\) are. 
    Hence we can apply Proposition \ref{sumofunits} to see that every element is a sum of at most \(3r-k\) elements of \(S(f)\), from which we apply Theorem \ref{adlerkonheim} to conclude the result.
\end{proof}

Furthermore, in light of Proposition \ref{sumofunits}, we expect that for \(N\) even, we can apply Proposition \ref{adlerkonheim} with \(v = 3\) to see that ninth-order data determines a function \(f: \ZZ_N \to \QQ\). 
However, Gr\"unbaum and Moore prove that only sixth-order data is needed. 
We offer in explanation Theorem \ref{mainthm:B}, a refinement of Proposition \ref{adlerkonheim}, from which the Gr\"unbaum-Moore theorem is an immediate consequence.

We first define some useful notation:
\begin{definition}
    Given a subgroup \(H \leq G\), we denote the set of products of \emph{exactly} \(v \in \NN\) elements of \(H\) by
    \[H^v \coloneqq \left\{\prod_{i=1}^v h_i : h_i \in H\right\}.\]
\end{definition}
The following lemma will be used throughout the proof of Theorem \ref{mainthm:B}:
\begin{lemma}\label{2 and k}
    Let \(f,g \in L^1(G,\RR)\) for \(G\) a locally compact abelian group. 
    If \(f\) and \(g\) have the same invariants of order \(2\) and \(k\), then 
    \[\prod_{i=1}^{m}\frac{\hat{f}(x_i)}{\hat{g}(x_i)} = \prod_{j=1}^{n}\frac{\hat{f}(y_j)}{\hat{g}(y_j)}\]
    for any \(x_i, y_j \in S(f)\) such that \(\prod_{i=1}^{m} x_i = \prod_{j=1}^{n} y_j\) and \(m + n = k\).
\end{lemma}
\begin{proof}
    Since \(\rho_2(f) = \rho_2(g)\) and \(f,g\) are real-valued, we have \(S(g) = S(f)\), so our fractions are well-defined, and we get
    \begin{align*}
        \dual{\rho}_2(f)(y,y^{-1}) &= \dual{\rho}_2(g)(y,y^{-1})\\
        \hat{f}(y)\hat{f}(y^{-1}) &= \hat{g}(y)\hat{g}(y^{-1})\\
        \implies \frac{\hat{g}(y^{-1})}{\hat{f}(y^{-1})} &= \frac{\hat{f}(y)}{\hat{g}(y)}
    \end{align*}
    Since \(\rho_k(f) = \rho_k(g)\), by treating \(\dual{\rho}_k\) as a function on \(k\) characters, we have
    \begin{align*}
        \dual{\rho}_k(f)(x_1, \dotsc, x_m, y_1^{-1}, \dotsc, y_n^{-1}) &= \dual{\rho}_k(g)(x_1, \dotsc, x_m, y_1^{-1}, \dotsc, y_n^{-1})\\
        \prod_{i=1}^{m} \hat{f}(x_i) \prod_{j=1}^{n} \hat{f}(y_j^{-1}) &= \prod_{i=1}^{m} \hat{g}(x_i) \prod_{j=1}^{n} \hat{g}(y_j^{-1})\\
        \implies \prod_{i=1}^{m}\frac{\hat{f}(x_i)}{\hat{g}(x_i)} &= \prod_{j=1}^{n}\frac{\hat{g}(y_j^{-1})}{\hat{f}(y_j^{-1})}
    \end{align*}
    Putting these together, we get the desired equality.
\end{proof}

\thmB*

We show below that the hypothesis is equivalent to there existing a ``special element'' \(a\in S(f)\) such that every element of \(\bangle{S(f)}\) is either the product of exactly \(v\) elements of \(S(f)\), or such a product times \(a\) (i.e. in the nontrivial coset of \(S(f)^v\)).
This is also equivalent to assuming that every element of \(\bangle{S(f)}\) is either the product of \emph{at most} \(v\) elements of \(S(f)\), or such a product times \(a\), since any product of fewer than \(v\) terms can be padded with pairs \(a^{-1}a=1\) until there are either exactly \(v\) terms, or \(v\) terms with an additional \(a\) at the end.
As a warning, this does \emph{not} mean that every element that is a product of at most \(v\) terms is a product of exactly \(v\) terms. 
In Gr\"unbaum and Moore's proof of Theorem \ref{grunbaum-moore}, \(v=2\) and \(a=e_1\) are used; our proof proceeds similarly using such a special element. 

For odd \(v\), this is actually no stronger than Theorem \ref{adlerkonheim}, since for any \(a \in S(f)\) as the special element,
  \(a^{-1}\) is always in \(S(f)\) by Theorem \ref{result} (because complex conjugation is always in the Galois group \(\Gal(\RR(\zeta_M)/\RR)\)), so 
\[\prod_{n=1}^{v} a^{(-1)^{n+1}} = aa^{-1}aa^{-1}\dotsm a = a \in S(f)^v.\]
Therefore, \(S(f)^v\) cannot be an index 2 subgroup for odd \(v\), so we may as well assume that \(v\) is even.

\begin{proof}[Proof of Theorem \ref{mainthm:B}]
    Suppose \(S(f)^v\) is an index 2 subgroup of \(\bangle{S(f)}\) for \(v \geq 2\) even. 
    We claim that there exists some \(a \in S(f)\) that lies in the nontrivial coset.
    
    For the sake of contradiction, assume otherwise, i.e., \(S(f) \subseteq S(f)^v\).
    Since we assumed \(S(f)^v\) is a group, we conclude \(\bangle{S(f)}\subseteq S(f)^v\), contradicting that \(|\bangle{S(f)}:S(f)^v|=2\).
    
    Suppose finally that \(\rho_n(f) = \rho_n(g)\) for \(n = 1, \dotsc, 3v\). 

    We first define a character of \(S(f)^v\) which extends the ratio \(\hat{f}/\hat{g}\) where defined. We will then extend this character to all of \(\bangle{S(f)}\) and conclude by the duality theorem. 
    
    To this end, for any \(x = x_1 x_2 \dotsm x_v \in S(f)^v\) with each \(x_i \in S(f)\), define 
    \[\lambda(x) = \prod_{i=1}^v \frac{\hat{f}(x_i)}{\hat{g}(x_i)}.\] 
    If \(x = \prod_{i=1}^v x_i = \prod_{i=1}^v y_i\) has two such representations, then \[\prod_{i=1}^v \frac{\hat{f}(x_i)}{\hat{g}(x_i)} = \prod_{i=1}^v \frac{\hat{f}(y_i)}{\hat{g}(y_i)}\] by Lemma \ref{2 and k}. 
    Hence \(\lambda\) is well-defined on \(S(f)^v\).
    Furthermore, it agrees with \(\frac{\hat{f}(x)}{\hat{g}(x)}\) if \(x \in S(f)\), since \[\frac{\hat{f}(x)}{\hat{g}(x)} = \prod_{i=1}^v \frac{\hat{f}(x_i)}{\hat{g}(x_i)}\] for \(x = \prod_{i=1}^v x_i\) follows from Lemma \ref{2 and k} and agreement of autocorrelation data of order \(v+1 \leq 3v\) (since \(v \geq 2\)). 
    Finally, we need to check that \(\lambda\) satisfies \[\lambda(xy) = \lambda(x)\lambda(y)\] (i.e., it is a homomorphism). 
    Let \(S(f)^v \ni x = \prod_i x_i\) and \(S(f)^v \ni y = \prod_j y_j\) for \(x_i, y_j \in S(f)\). 
    Since \(S(f)^v\) is a subgroup, \(xy \in S(f)^v\), so write \(xy = \prod_k z_k\) with each \(z_k \in S(f)\). 
    Then 
    \begin{equation}
        \label{lambda(xy)}
        \lambda(xy) = \lambda\left(\prod_{k=1}^v z_k\right) = \prod_{k=1}^v \frac{\hat{f}(z_k)}{\hat{g}(z_k)},
    \end{equation}
    and 
    \begin{equation}
        \label{lambda(x)lambda(y)}
        \lambda(x)\lambda(y) = \prod_{i=1}^v\frac{\hat{f}(x_i)}{\hat{g}(x_i)} \prod_{j=1}^v \frac{\hat{f}(y_j)}{\hat{g}(y_j)},
    \end{equation}
    and equality of the expressions in (\ref{lambda(xy)}) and (\ref{lambda(x)lambda(y)}) follows from equality of \(3v\)th-order autocorrelations and Lemma \ref{2 and k}. 
    After the straightforward checks \[|\lambda(x)| = \left|\frac{\hat{f}(x)}{\hat{g}(x)}\right| = 1\] by second-order data and \[\lambda(e_0) = \frac{\hat{f}(e_0)}{\hat{g}(e_0)} = 1\] by first-order data, we conclude that \(\lambda\) is a character of \(S(f)^v\). 

    To extend \(\lambda\) to all of \(\bangle{S(f)}\), it suffices to define it on elements of the form \(xa, x \in S(f)^v\). 
    We define it in the obvious way: \[\lambda(xa) = \lambda(x) \frac{\hat{f}(a)}{\hat{g}(a)}.\] 
    There is no conflict with the previous definition since we are defining \(\lambda\) on the nontrivial coset \(S(f)^va\). 
    Our new definition agrees with the ratio \(\frac{\hat{f}}{\hat{g}}\) on \(S(f)\) since, if \(x \in S(f)\) can be written \(a\cdot \prod_ix_i\) with each \(x_i \in S(f)\), then \[\lambda(x) = \left[\prod_{i=1}^v \frac{\hat{f}(x_i)}{\hat{g}(x_i)}\right] \frac{\hat{f}(a)}{\hat{g}(a)} = \frac{\hat{f}(x)}{\hat{g}(x)}\] by equality of autocorrelations through order \(v+2 \leq 3v\), and Lemma \ref{2 and k}. 
    Since \(\left|\frac{\hat{f}(a)}{\hat{g}(a)}\right| = 1\), we still have \(|\lambda(x)| = 1\) for all \(x\). 
    It remains only to check that \(\lambda\) is a homomorphism. 

    Let \(x, y \in \bangle{S(f)}\). 
    There are three cases to consider. 
    \begin{enumerate}
        \item 
        Both \(x, y\) lie in \(S(f)^v\). In this case, \(\lambda(xy) = \lambda(x)\lambda(y)\) by the previous argument.

        \item 
        Exactly one of the two, say \(x\), lies in the nontrivial coset \(S(f)^v a\). 
        Write \(x = a\cdot \prod_{i=1}^v x_i\) and \(y = \prod_{j=1}^v y_j\) with the \(x_i, y_j \in S(f)\). 
        Set \(y' = \prod_{j=1}^v y_j\), so that \(y = y'a\). 
        Then 
        \[\lambda(xy) = \lambda(xy'a) = \lambda(xy')\frac{\hat{f}(a)}{\hat{g}(a)}\]
        since \(xy' \in S(f)^v\) as \(S(f)^v\) is a subgroup. 
        But then, since \(\lambda\) acts homomorphically on \(S(f)^v\), 
        \[\lambda(xy')\frac{\hat{f}(a)}{\hat{g}(a)} = \lambda(x)\lambda(y')\frac{\hat{f}(a)}{\hat{g}(a)} = \lambda(x)\lambda(y'a) = \lambda(x)\lambda(y)\] 
        by definition.

        \item 
        Lastly, suppose both \(x, y \in S(f)^va\). 
        Observe first that \(a^2 = a^2(a^{-1}a)^{(v-2)/2} \in S(f)^v\) since \(v\) is even. 
        (Of course, \(a^2 \in S(f)^v\) when \(v\) is odd, too, as \(a \in S(f)^v\) in that case.) 
        Furthermore, using the definition of \(\lambda\) on that subgroup, 
        \[\lambda(a^2) = \frac{\hat{f}(a)}{\hat{g}(a)} \cdot \frac{\hat{f}(a)}{\hat{g}(a)} \cdot \left(\frac{\hat{f}(a)}{\hat{g}(a)} \cdot \frac{\hat{f}(a^{-1})}{\hat{g}(a^{-1})}\right)^{(v-2)/2}.\] 
        By equality of first- and second-order data, 
        \[\frac{\hat{f}(a)}{\hat{g}(a)} \cdot \frac{\hat{f}(a^{-1})}{\hat{g}(a^{-1})} = \frac{\hat{f}(1)}{\hat{g}(1)} = 1,\] 
        so that 
        \[\lambda(a^2) = \left(\frac{\hat{f}(a)}{\hat{g}(a)}\right)^2.\]

        Now, write \(x=x'a\) and \(y = y'a\) with \(x' = \prod_i x_i \in S(f)^v\) and \(y' = \prod_j y_j \in S(f)^v\). Then \[\lambda(x)\lambda(y) = \lambda(x')\frac{\hat{f}(a)}{\hat{g}(a)} \cdot \lambda(y')\frac{\hat{f}(a)}{\hat{g}(a)} = \lambda(x')\lambda(y') \left(\frac{\hat{f}(a)}{\hat{g}(a)}\right)^2 \] 
        \[= \lambda(x'y')\left(\frac{\hat{f}(a)}{\hat{g}(a)}\right)^2 = \lambda(x'y')\lambda(a^2) = \lambda(x'y'a^2) = \lambda(xy)\]
        since \(\lambda\) is a homomorphism on \(S(f)^v \ni a^2, x', y'\).
    \end{enumerate}

    We conclude that \(\lambda\) is a character of \(\bangle{S(f)}\) extending the ratio \(\hat{f}/\hat{g}\) where defined. 
    Hence \[\hat{f}(e) = \lambda(e)\hat{g}(e)\] for all \(e \in \dual{G}\), either because both sides are \(0\) or because \(\lambda(x) = \hat{f}(x)/\hat{g}(x)\). 
    But, by the duality theorem, \(\lambda(e) = e(t)\) for some fixed \(t \in G\), and consequently \[\hat{f}(e) = e(t)\hat{g}(e) \hspace{1cm} \forall e \in \dual{G}\] from which we conclude \(f(x) = g(x+t)\) for all \(x \in G\) by Fourier inversion. 
\end{proof}

Theorem \ref{mainthm:B} turns the ``obvious'' argument in the \(\ZZ_N\) setting into a sharp bound, where existing results (i.e. that of Adler and Konheim) offer provably suboptimal bounds. This sharpness is essential to working effectively with autocorrelation data, as the order of data available is frequently limited.

\section{Acknowledgments}

This work was supported in part by the NSF grant for Algebraic Geometry and Representation Theory at Northeastern University DMS–1645877.

\printbibliography%[heading=bibintoc]
\end{document}